\documentclass[10pt]{amsart}
\pdfoutput=1
\usepackage[utf8]{inputenc}
\usepackage[T1]{fontenc}
\usepackage{amssymb}
\usepackage[dvipsnames,svgnames,x11names,hyperref]{xcolor}
\usepackage{amsmath, amsfonts,amsthm}
\usepackage[colorlinks=true]{hyperref}
\usepackage[all]{xy}
\usepackage{xypic}
\numberwithin{equation}{subsection}
\theoremstyle{plain}
\newtheorem{thm}[subsection]{Theorem}
\theoremstyle{definition}
\newtheorem{defn}[subsection]{Definition}
\newtheorem{example}[subsection]{Question}
\newtheorem{remark}[subsection]{Remark}
\theoremstyle{plain}
\newtheorem{prop}[subsection]{Proposition}
\newtheorem{lemma}[subsection]{Lemma}
\theoremstyle{cor}
\newtheorem{cor}[subsection]{Corollary}
\theoremstyle{definition}
\theoremstyle{remark}
\usepackage{geometry}
\fontfamily{ptm}
\selectfont
\usepackage{mathrsfs}
\usepackage{tikz-cd}
\newcommand{\X}{\mathcal{X}}
\newcommand{\Y}{\mathcal{Y}}
\DeclareMathOperator{\id}{id}
\DeclareMathOperator{\A}{\mathbb{A}}
\DeclareMathOperator{\N}{\mathbb{N}}
\DeclareMathOperator{\Z}{\mathbb{Z}}
\DeclareMathOperator*{\Q}{\mathbb{Q}}

\newcommand{\OO}{\mathscr{O}}
\DeclareMathOperator{\Hom}{Hom}

\DeclareMathOperator{\chr}{char}

\DeclareMathOperator{\Aut}{Aut}
\DeclareMathOperator{\Autt}{\mathbf{Aut}}

\DeclareMathOperator{\Isom}{\mathbf{Isom}}
\DeclareMathOperator{\Homm}{\mathbf{Hom}}
\DeclareMathOperator{\Hilb}{Hilb}
\DeclareMathOperator{\Spec}{Spec}
\DeclareMathOperator{\Spf}{Spf}

\DeclareMathOperator{\Def}{Def}

\newcommand{\RR}{\text{R}}

\DeclareMathOperator{\Sch}{\mathsf{Sch}}

\newcommand{\shT}{\mathscr{T}}
\DeclareMathOperator{\Art}{\mathsf{Art}}

\DeclareMathOperator{\pr}{pr}
\newcommand{\Ext}{\mathcal{E}xt}
\newcommand{\XX}{\mathcal{X}}
\newcommand{\Sets}{\mathsf{Set}}

\newcommand{\dd}{\textrm{d}}
\newcommand{\UU}{\mathcal{U}}
\DeclareMathOperator{\ch}{char}

\DeclareMathOperator{\PP}{\mathbb{P}}

\DeclareMathOperator{\Coh}{\textsf{Coh}}

\newcommand{\I}{\mathcal{I}}
\DeclareMathOperator{\colim}{colim}
\DeclareMathOperator{\trdeg}{tr.deg}
\DeclareMathOperator{\KS}{KS}
\DeclareMathOperator{\Hilbb}{\mathbf{Hilb}}
\DeclareMathOperator{\Isomm}{Isom}
\DeclareMathOperator{\Proj}{Proj}

\title{An algebraic variant of the Fischer--Grauert Theorem}
\author{Paweł Poczobut}
\address{Institute of Mathematics, University of Warsaw, ul. Banacha 2, 02-097 Warszawa, Poland}
\email{p.poczobut@student.uw.edu.pl}
\begin{document}
\hypersetup{linkcolor=Teal}
\hypersetup{citecolor=Teal}
\hypersetup{urlcolor=Teal}
\maketitle
\begin{abstract}
	 A well-known theorem of W. Fischer and H. Grauert states that analytic fiber spaces with all fibers isomorphic to a fixed compact connected complex manifold are locally trivial. Motivated by this result, we show that if $k$ is an algebraically closed field of infinite transcendence degree over its prime field, then every smooth projective family over a reduced $k$-scheme of finite type with isomorphic fibers having reduced automorphism group schemes is locally trivial in the \'etale topology. We do so by reducing the problem to the case when the base is a smooth integral curve, and then, using the vanishing of the Kodaira--Spencer map, we prove formal triviality of such families at every geometric point of the base. We also provide examples of smooth projective fibrewise trivial families in positive characteristic whose Kodaira--Spencer map are nowhere vanishing.
\end{abstract}
\section{Introduction}
Suppose $f\colon\XX\to S$ is a morphism of complex manifolds such that for every $s\in S$, the fibre $\XX_s:=f^{-1}(s)$ is analytically isomorphic to a fixed complex manifold $X$. It is natural to investigate to what extent the family $f$ is, locally on $S$, isomorphic to the trivial family. In 1965, Wolfgang Fischer and Hans Grauert gave the following answer to this question.

\begin{thm}[{\cite{FG}}]\label{fischergrauert}
	Let $f :\XX\rightarrow S$ be a proper submersion of complex manifolds. Suppose that the fibres $\XX_s$ are analytically isomorphic to a fixed complex manifold $X$. Then, $f$ is locally trivial, i.e. for every $s\in S$ there exists an open neighborhood $s\in U_s\subset S$ such that $\XX_{U_s}:=f^{-1}(U_s)$ and $X\times U_s$ are isomorphic as analytic fibre bundles over $U_s$.
\end{thm}

Generalizations of this result are known in the context of complex spaces (in the sense of Grauert). An analogous statement when $f\colon\XX\to S$ is a flat proper morphism of complex spaces with $S$ reduced was proved Schuster \cite[Satz 4.9]{schuster}. Other results may also be found e.g. in a work of Donin \cite[3.1, 3.2]{donin}.
\medskip

One may expect that a theorem of this kind should have an analog in algebraic geometry. In this paper, we show that it is indeed the case. Before we state our main result, let us introduce some terminology that we use throughout this work.
\begin{defn}\label{fibtrivial}
	Let $k$ be an algebraically closed field and let $X$ be a $k$-scheme. A flat morphism of finite type $k$-schemes $f\colon\XX\rightarrow S$ is \textbf{fibrewise trivial} (with fibre $X$), if there exists an isomorphism $\XX_s=\XX\times_S s\simeq X$ for every $s\in S(k)$. A fibrewise trivial family $f$ with fibre $X$ is \textbf{trivial} if $\XX$ and $X\times_k S$ are isomorphic as $S$-schemes. A family $f$ is {\bf Zariski-locally trivial} if for every $s\in S(k)$ there exists an open neighborhood $s\in U\subset S$ such that $f_U\colon\XX_{U}\to U$ is trivial. We say that a morphism of finite type $k$-schemes $f\colon\XX\rightarrow S$ is \textbf{isotrivial at} $s\in S(k)$ if there exists an étale neighborhood $V\rightarrow S$ of $s$ such that $f_V\colon\XX\times_S V\rightarrow V$ is trivial. We call $f$ \textbf{isotrivial} if it is isotrivial at every $s\in S(k)$.
\end{defn}

It is well-known that an analog of Theorem \ref{fischergrauert} does not hold if we interpret the concept of local triviality in the naïve manner, i.e. a fibrewise trivial family $f\colon\XX\to S$ may not be Zariski-locally trivial. A standard example is an elliptic scheme $p\colon\mathscr{E}\subset \PP^2_S=\Proj \OO_S[X,Y,Z]\rightarrow S$ cut out by the equation $Y^2Z-X^3-tZ^3=0$. Clearly $\XX/S$ is fibrewise trivial. 
The generic fibre $\mathscr{E}_\eta$ of $\mathscr{E}$ is an elliptic curve over $\mathbb{C}(t)$ defined by the equation $Y^2Z=X^3-tZ^3$ and it is not isomorphic to the base change $\mathscr E_1\times_\mathbb C \mathbb C(t)$ of the fibre of $p$ \cite[Exercise 4.D]{Olsson}\footnote{But they become isomorphic after passing to the extension $\mathbb{C}(t^{\frac 16})$ of $\mathbb C(t)$.}, so $p$ isn't Zariski-locally trivial.\medskip

Although the Zariski topology is too coarse for our purposes, the above example suggests that we could obtain more satisfying results when working in finer Grothendieck topologies, in particular in the \'etale topology of the base of considered family. 
From Definition \ref{fibtrivial}, it follows that when $S$ is connected, an isotrivial family $f\colon\XX\to S$ is fibrewise trivial. In this paper, we investigate whether the converse statement holds under some reasonable assumptions about the family $f$. This topic appears to be well-known to experts. In \cite[Remark 1.4]{kovacs}, the author mentions briefly the fact that projective families over $\mathbb{C}$ with isomorphic fibers are \textit{generically} locally trivial with respect to the étale topology\footnote{Note, however, that the notion of isotriviality in op.cit. differs from the one being used throughout this paper: compare \cite[Definition 1.2]{kovacs} with Definition \ref{fibtrivial}.}. In \cite[Section 1]{BBG}, authors claim that an algebraic variant of Theorem \ref{fischergrauert} should follow from \cite[Corollary 2.6.10]{sernesi} which states that isotriviality is equivalent to formal triviality (\ref{sef:formal-triviality}) for flat projective families.

That being said, none of the sources listed above provides a precise statement of a variant of Theorem \ref{fischergrauert}. Nor is it immediately clear what assumptions about a fibrewise trivial family $f\colon\XX\to S$ are sufficient to deduce formal triviality of $f$ at every closed points of its base. To the author's knowledge, both an explicit statement and a proof of such a result appear to be missing from the literature. In what follows, we attempt to fill the gap by providing a precise statement of an algebraic variant of Theorem \ref{fischergrauert} (Theorem \ref{mainthm}) for fibrewise trivial families of projective $k$-varieties in the case when $k$ is a \emph{sufficiently large} algebraically closed field, and all varieties and morphisms are algebraic, together with a detailed proof of this assertion. Namely, our main result is the following theorem. 

\begin{thm}\label{mainthm}
	Let $k$ be an algebraically closed field of infinite transcendence degree over its prime field. Suppose that $f\colon\XX\to S$ is a smooth, projective, fibrewise trivial morphism of finite type $k$-schemes with fibre $X$. Suppose furthermore that $S$ is reduced, and that the automorphism group scheme (Proposition \ref{prop:isomrepresentable}) $\Aut^{X}_k$ of the fibre is reduced. Then $f$ is isotrivial.
\end{thm}

Note that the assumption about $\Aut^X_k$ being reduced is automatic when $\chr k=0$ \cite[Theorem 3.23]{milne}.\medskip

Let us present a brief outline of the proof of Theorem \ref{mainthm} when $\chr k=0$. First, we prove generic isotriviality (Theorem \ref{genericisotriviality}), a minor generalization of the result mentioned in \cite{kovacs}, where the author considers only schemes over complex numbers. This is the only step in the proof where we make use of the condition $\trdeg (k/\mathbb{F})=\infty$ for $\mathbb F$ being the prime field of $k$. We use this result to show that the relative Kodaira--Spencer map $\KS_f:\mathscr{T}_{S/k}\rightarrow\RR^1f_\ast\shT_{\XX/S}$ associated to $f$ vanishes on $S$ (Proposition \ref{kodairavanishes}). Next, we reduce Theorem \ref{mainthm} to the situation when $S$ is a smooth integral curve of finite type over $k$, using a variant of the valuative criterion of flatness (Proposition \ref{flatbccurves}). Finally (Section \ref{mainresult1}), we prove formal triviality at every closed point of the base with an explicit calculation, using the vanishing of $\KS_f$.
\medskip

The strategy behind the proof of Theorem \ref{mainthm} in positive characteristic (Section \ref{mainresult2}) is the same except for a few caveats --- the final calculation done in the proof of the $\chr k=0$ case relies heavily on this assumption. Hence, assuming $\chr k=p>0$, we prove formal triviality using the fact that now our schemes are equipped with Frobenius morphism. We apply a variant of Ogus' lemma (Proposition \ref{ogus-refinement}), which originally states that the vanishing of the Kodaira--Spencer map guarantees that a smooth proper family $\XX\to S=\Spec k[[t]]$ descends along Frobenius $F_S:S\to S$ (Lemma \ref{frobeniusdescend}). In our setting, when $S$ is a smooth curve over $k$, we show that the descent can be done in an appropriate \'etale neighborhood of any given point $s\in S(k)$. We deduce this fact from the original result of Ogus by approximating the scheme $\Spec \widehat\OO_{S,s}$ with smooth $S$-schemes \cite[Tag \href{https://stacks.math.columbia.edu/tag/07GC}{\texttt{07GC}}]{stacks}.\medskip

Lastly, we show how to construct a fibrewise trivial family with nowhere vanishing Kodaira--Spencer map (Section \ref{KSnotzero}), thereby proving that $\XX/S$ may not even be generically isotrivial in the case when $\Aut^X_k$ is non-reduced.\medskip

We think that an interesting question that we didn't find an answer to is whether we can remove the assumptions about the scheme $\Aut^X_k$ by working with fppf topology.
\begin{example}
	Let $k$ be an algebraically closed field of infinite transcendence degree over its prime field. Suppose that $f\colon\XX\rightarrow S$ is a smooth, projective, fibrewise trivial morphism of finite type $k$-schemes with $S$ reduced. Does it imply that $f$ is fppf-locally trivial?
\end{example}
\section*{Acknowledgements}
I am indebted to Piotr Achinger for suggesting me the topic of this thesis, for his infinite patience when supervising the project, and for many fruitful discussions, which taught me a lot of mathematics.\:I would like to thank Adrian Langer for many helpful suggestions. Many thanks to Joachim Jelisiejew for help with the proof of Lemma~\ref{curvesareenough}. A part of this work concerning the characteristic~$0$ case was the author's Bachelor Thesis at Warsaw University. The author was supported by NCN Sonata grant number 2017/26/D/ST 1/00913.
\section{Preliminaries}
\subsection{Notation} Throughout, $k$ always denotes an algebraically closed field. Given a scheme $S$, we denote the category of $S$-schemes by $\Sch_S$. When $S=\Spec A$ is affine, we write $\Sch_A$ for $\Sch_{\Spec A}$. We often write $\XX/S$ for an $S$-scheme $\XX\to S$, provided that the structure morphism is clear from the context. We given $\XX/S$ and $\phi:T\to S$, we write $\phi^\ast\X$ for a choice of pullback of $\XX$ along $\phi$. When the structure morphism  $T\to S$ is clear, we also denote $\phi^\ast \XX$ with $\XX_T$.\medskip

Given a scheme $S$, we denote with $\Coh(S)$ the category of coherent sheaves on $S$. When $\mathscr F\in \Coh(S)$ and $s\in S$ is a point, we denote with $\mathscr F_s$ the stalk of $\mathscr F$ at $s$ and by $\mathscr F(s)$ the fibre of $\mathscr F$ at $s$, i.e. $\mathscr F(s)=\mathscr F_s/\mathfrak m_s\mathscr F_s$, where $\mathfrak m_s\subset \OO_{S,s}$ is the maximal ideal.\medskip

A few times, e.g. in the proof of generic isotriviality (Theorem \ref{genericisotriviality}) and of \'etale local Frobenius descent (Proposition \ref{ogus-refinement}) for families with vanishing Kodaira--Spencer map, we will make use of the fact that a finite type morphism $f\colon \XX\to\Spec A$ can be described with finitely many equations only, hence can be defined over a finitely generated $\Z$-algebra $B\subset A$, see \cite[II.5.10]{Kollar}.
\begin{defn}\label{models}
Suppose $S=\Spec A$ is affine, let $B\subset A$ be a subring, and let $\XX\in \Sch_S$. We say that $\XX$ \textbf{is defined over $B$} if there exists a \textbf{model of $\XX$ over $B$}, i.e. $\XX^{(B)}\in \Sch_B$ of finite type such that $\X$ is isomorphic to $\X^{(B)}\times_{\Spec B} S$ as $S$-schemes.
\end{defn}
\subsection{$\Hom$, $\Isomm$ and $\Aut$ schemes} Let $S$ be a $k$-scheme and $\X$, $\Y$ be two $S$-schemes. We define the functors $\Homm^{\X,\Y}_S$ of morphisms and $\Isom^{\X,\Y}_S$ of isomorphisms between $\X$ and $\Y$. Using the theory of Hilbert scheme \cite{gro4}, one can show that these functors are representable by a locally finite type $S$-schemes $\Hom^{\X,\Y}_S$ and $\Isomm^{\X,\Y}_S$ respectively. The existence of $\Isomm^{\X,\Y}_S$ will be the main tool in our proof of Theorem \ref{mainthm} and we will use the existence of $\Hom^{\X,\Y}_S$ in Section \ref{KSnotzero}.
\begin{defn}\label{def:isomandaut}Suppose that $\X$ and $\Y$ are $S$-schemes. 
We define the functor of morphisms $\Homm^{\X,\Y}_S:\Sch_S\rightarrow \Sets$ as follows: given $T\in \Sch_S$ we set\begin{displaymath}
	\Homm^{\X,\Y}_S(T)=\left\{\phi:\X_T\longrightarrow \Y_T\:\: \bigg\rvert\:\:\substack{\text{\normalsize $\phi$ is an morphism} \\ \text{\normalsize of $T$-schemes}}\right\}.\end{displaymath}
We also define the functor of isomorphisms $\Isom^{\X,\Y}_S:\Sch_S\rightarrow \Sets$ as follows: given $T\in \Sch_S$ we set\begin{displaymath}
	\Isom^{\X,\Y}_S(T)=\left\{\phi:\X_T\longrightarrow \Y_T\:\: \bigg\rvert\:\:\substack{\text{\normalsize $\phi$ is an isomorphism} \\ \text{\normalsize of $T$-schemes}}\right\}.
\end{displaymath}
Also, we define $\Autt^{\X}_S:=\Isom^{\X,\X}_S$.
\end{defn}

Let $\X\in \Sch_S$. Then we define the  Hilbert functor $\Hilbb^{\XX}_S:\Sch_S\to \Sets$ associating to $S'\in \Sch_S$ the set of $S'$-flat closed subschemes of $\XX\times_S S'$ \cite[p. 265]{gro4}. When $\XX/S$ is \emph{projective}, then $\Hilbb^{\X}_S$ is representable by an $S$-scheme, denoted $\Hilb^{\XX}_S$ (op.cit.).

\begin{prop}[{\cite[4.c]{gro4}}]\label{prop:isomrepresentable}
	Suppose $\X/S$ and $\Y/S$ are projective. Then $\Homm^{\X,\Y}_S$ and $\Isom^{\X,\Y}_S$ are representable by schemes, locally of finite type over $S$. We denote these scheme with $\Hom^{\X,\Y}_S$ and $\Isomm^{\X,\Y}_S$ respectively. What's more, the obvious functor $\Isom^{\X,\Y}_S\to \Homm^{\X,\Y}_S$ induces an open embedding  $\Isomm^{\X,\Y}_S\to \Hom^{\X,\Y}_S$ of $S$-schemes and similarly the functor $\Homm^{\X,\Y}_S\to \Hilbb^{\X\times_S \Y}_S$ associating morphisms to their graphs \cite[4.6.6]{sernesi} induces an open embedding $\Hom^{\X,\Y}_S\to \Hilb^{\X\times_S \Y}_S$ of $S$-schemes.
\end{prop}
\begin{cor}
	Suppose that $\X/S$ is projective. Then the functor $\Autt^{\X}_S$ is representable by a group $S$-scheme, denoted $\Aut^{\X}_S$.
\end{cor}
Now suppose $\XX\rightarrow S$ is projective and fix a relatively very ample line bundle $\OO_{\XX}(1)$ on $\XX$. Let $\mathscr{F}\in\Coh(\XX)$. Given a point $s\in S$, write $\mathscr{F}_s$ for $\mathscr{F}_{\XX_s}$. Then, $\OO_{\XX}(1)$ restricts to a very ample line bundle $\OO_{\XX_s}(1)$ on the projective $k(s)$-scheme $\XX_s$. Thus, for every point $s\in S$, we obtain a Hilbert polynomial $p_{\mathscr{F}_{s}}(m)$. In case $\mathscr{F}=\OO_{\mathcal Z}$ for some subscheme $\mathcal{Z}\subset \XX$, we denote $p_{\OO_{\mathcal Z_s}}$ with $p_{\mathcal Z_s}$.\medskip

Consider the function $s\mapsto p_{\mathscr{F}_s}(m)$ for $s\in S(k)$. In case $\mathscr{F}$ is $S$-flat, this function is locally constant on $S$. Given $P\in \Q[y]$, we define the subfunctor $\Hilbb^{\XX,P}_S$ of $\Hilbb^{\XX}_S$ as follows
	\begin{displaymath}
		\Hilbb_{S}^{\X,P}(T)=\left\{\mathcal{Z}\subset \XX\times_S T\:\bigg\rvert\: \substack{\text{\normalsize $\mathcal{Z}$ is a $T$-flat closed subscheme of $\XX\times_S T$,}\\[0.5 em] \text{\normalsize and  $p_{\mathcal{Z}_t}=P$ for all $t\in T$}}\right\}.
	\end{displaymath}
	Then \begin{equation}\label{eq:hilbdisjointunion}
	\Hilbb^{\XX}_S = \bigsqcup_{P\in\Q[y]}\Hilbb^{\XX,P}_S
\end{equation}
and when $\XX/S$ is projective, each $\Hilbb^{\XX,P}_S$ is representable by a \emph{projective} $S$-scheme.
\subsection{}\label{componentsofisom} As a result, when $\X/S$ and $\Y/S$ are projective, by Proposition \ref{prop:isomrepresentable} and \eqref{eq:hilbdisjointunion}, we find that $\Isomm^{\X,\Y}_S=\bigsqcup_{n\in \N}I_n$ is a countable disjoint union with each $I_n$ quasi-projective over $S$. In particular, for $p:\Isomm^{\X,\Y}_S\to S$ denoting the structure morphism, we deduce from the Chevalley Theorem that each $p(I_n)$ is constructible in $S$.
\subsection{Exponential law}\label{exponentiallaw}
Let $\Homm^{\X,\Y}_S$ be the functor as in Definition \ref{def:isomandaut}. Then, treating schemes as functors of points, we find that we have an isomorphism
\begin{equation}\label{eq:explaw1}
	\Hom_{S}(T,\Homm^{\XX,\Y}_S)\xrightarrow{\sim} \Hom_{S}(T\times_S\XX, \Y)=\Hom_{T}(T\times_S\XX, T\times_S \Y)
\end{equation}
which is natural in $T\in \Sch_S$. In \eqref{eq:explaw1}, $\Hom_S$ denotes morphisms in the category of sheaves on $\Sch_S$ with fppf topology. When $\X$ and $\Y$ are projective over $S$, then $\Homm^{\X,\Y}$ is representable by an $S$-scheme so \eqref{eq:explaw1} induces a bijection
\begin{equation}\label{eq:explaw}
	\Hom_{S}(T,\Hom^{\XX,\Y}_S)\xrightarrow{\sim} \Hom_{S}(T\times_S\XX, \Y)
\end{equation}
between sets of morphisms of $S$-schemes. We call it the \textbf{exponential law}.

\subsection{Deformation theory} Let $k$ be an algebraically closed field. Denote $\Art_k$ the full subcategory of the category of $k$-algebras, consisting of local Artinian $k$-algebras with residue field equal to $k$.
\begin{defn}[{see \cite[1.2.1]{sernesi}}]
	Let $X\in \Sch_k$. A \textbf{deformation of $X$ over $A\in \Art_k$} is a Cartesian diagram
	\begin{equation*}
		\begin{tikzcd}
			X\ar{d}\ar{r}{j} & \XX\ar{d}{\pi}\\[1.5 em] \Spec k\ar{r} & \Spec A
		\end{tikzcd}
	\end{equation*}
	with $\pi$ flat.	
\end{defn}
We denote the set of isomorphism classes of deformations of $X$ over $A$ by $\Def_X(A)$. The association $A\mapsto \Def_X(A)$ defines a functor $\Def_X\colon\Art_k\to \Sets$. In particular, notice that here, a deformation is the data of an $A$-scheme $\XX$ \emph{and} an isomorphism $X\xrightarrow\sim\XX\times_A k$.
\begin{defn}{\cite[p. 44]{sernesi}}\label{pro-rep-definition}
	Suppose $X\in \Sch_k$. We say that $\Def_X$ is \textbf{pro-representable} if $\Def_X(\cdot)=\Hom_{k{\rm- alg}}(R,\cdot)$ for some complete local noetherian $k$-algebra $R$ with residue field $k$. We say $X\in \Sch_k$ \textbf{has pro-representable deformations} when $\Def_X$ is pro-representable.
\end{defn}
The common tool for determining whether $X$ has pro-representable deformations is \textit{Schlessinger's criterion} \cite[Theorem 2.11]{schless}. Standard examples of schemes with pro-representable deformations include smooth projective connected curves over $k$ \cite[Corollary 2.6.10]{sernesi} and abelian varieties \cite[Theorem 2.2.1]{oort}.

\section{Generic isotriviality}
The first step in our proof of Theorem \ref{mainthm} is showing that our claim holds generically on (i.e., over an open dense subset of)  $S$. The proof of this part relies on the following well-known lemma from the theory of constructible sets.

\begin{prop}\label{countablecover}
	Let $k$ be an algebraically closed field of infinite transcendence degree over its prime field $\mathbb{F}$, and let $S$ be an irreducible $k$-scheme of finite type. Suppose that $\mathcal{C}=\{C_i\}_{i\in \N}$ is a countable family of constructible subsets of $S$, satisfying $S(k)\subset \bigcup_{i\in \N}C_i(k)$. Additionally, assume that there exists an intermediate extension $\mathbb{F} \subset \ell\subset k$ with $\trdeg (\ell/ \mathbb F)<\infty$ such that $S$ admits a model $S^{\ell}$ over $\ell$ and all $C_i\in \mathcal{C}$ admit models $C_i^\ell$ which are constructible subsets of $S^\ell$.
	Then there exists $C\in\mathcal{C}$ dense in $S$.
\end{prop}

\begin{proof}
Let $\eta\in S$ be the generic point of $S$. Every constructible set is a union of finitely many constructible irreducible sets. Therefore, by replacing each $C_i$ with the collection of its irreducible components, we may assume that $\mathcal{C}$ is a countable family of constructible irreducible subsets of $S$. If $\dim S=0$ then $|S|$ is a single point so there is nothing to prove. Hence we may suppose $\dim S\geq 1$. We prove the assertion in three steps.\medskip

\textbf{Step 1:} Suppose that $S=\A_k^1=\Spec k[T]$.
Replacing $\ell$ with $\overline{\ell}$, we may assume that $\ell$ is algebraically closed (preserving the condition $\trdeg (\ell/\mathbb F)<\infty$). By assumption there exist constructible subsets $C_{i}^\ell\subset \A^1_{l}$ such that $C_i = C_i^\ell\otimes_{\ell}k$. If none of $C_i$ is dense in $S$, then each $C_i$ is a finite union of closed points, thus so is each $C_i^\ell$. For $\lambda\in k\setminus \ell$, let $\mathfrak{m}=(T-\lambda)$ and suppose $p=[\mathfrak{m}]\in \A^1_k$ belongs to some $C_i$. Then the image of $p$ in $\A^1_\ell$ is the generic point of $\A^1_\ell$, hence $C^\ell_i$ is dense, a contradiction.\medskip

\textbf{Step 2:} $S=\A^d_k$. We proceed by induction on $d$. The base of induction $d=1$ is done in \textbf{Step~1}. 
Let $d>1$ and suppose $\A^{r}_k$ satisfies the assertion of the proposition for all $1\leq r<d$. Suppose that none of the $C_i$'s contain $\eta$. Then $\dim C_i\leq d-1$ for all $i\in\N$. Since every $C_i$ is assumed to be irreducible, there exists at most one hyperplane $H_i\subset S$ contained in $\overline{C_i}$.
Let $H\subset \A^d_{k}$ be a hyperplane different from $H_i$'s (such an $H$ exists because each $H_i$ is defined over $\ell$ hence it suffices to choose $H$ which is not defined over $\ell$) with generic point $\xi$. By our assumptions, $\overline{C_i}\cap H\subsetneq H$ is a proper closed subset of $H$, so $\xi\notin C_i\cap H$. The family $\mathcal{C}|_{H}:= \{C_i\cap H\}_{i\in\N}$ consists of countably many constructible subsets of $H$, with none of $C_i\cap H$ dense in $H$. By the inductive assumption, $H(k)\not\subset \bigcup_{i\in\N}( C_i\cap H)$. Then, for every $p\in H(k)\setminus\left(\bigcup_{i\in\N}( C_i\cap H)\right)$ we have $p\in S(k)\setminus \left(\bigcup_{i\in\N} C_i\right)$, a contradiction.
\medskip

	 \textbf{Step 3:} Let $S$ be an arbitrary irreducible $k$-scheme of dimension $d\geq 1$. We may assume that $S$ is affine. By Noether's Normalization Lemma \cite[Tag \href{https://stacks.math.columbia.edu/tag/0CBI}{\texttt{0CBI}}]{stacks}
	  there exists a finite surjective map $\pi :S\rightarrow \mathbb{A}^d_k$. For the generic point $\mu\in\A^d_k$, we have $\eta=\pi^{-1}(\mu)$. Write $Z_i=\pi(C_i)$. Each $Z_i$ is irreducible and by Chevalley's Theorem it is also constructible. Surjectivity of $\pi$ implies $\A_k^d(k)\subseteq \bigcup_{i\in\N} Z_i$. By \textbf{Step 2}, there exists $i_0\in \N$ such that $Z_{i_0}$ is dense in $\A^d_k$, meaning that $\mu\in Z_{i_0}$. Then $\eta\in C_{i_0}$, proving that $C_{i_0}$ is dense in $S$.
\end{proof}

We apply this result to the scheme of isomorphisms $\Isomm^{\XX,X\times S}_S$ to show that one of its components dominates $S$.

\begin{thm}\label{genericisotriviality}
Suppose that $k$ is algebraically closed, of infinite transcendence degree over its prime field $\mathbb{F}$. Let $f:\XX\rightarrow S$ be a smooth, projective, fibrewise trivial morphism of finite type $k$-schemes with $S$ reduced and with fibre $X$. Then there exists an open dense subset $U$ of $S$ such that $f_U:\XX_U\to U$ is fppf-locally trivial. If we assume furthermore that the group scheme $\Aut_k^X:=\Isomm^{X,X}_k$ is reduced, then $\XX_U/U$ is isotrivial.
\end{thm}
\begin{proof}
It is enough that we restrict ourselves to the case when $S$ is affine. Let $\mathbb{F}\subset \ell\subset k$ be an algebraically closed subfield of $k$ satisfying $\trdeg(\ell/\mathbb{F})<\infty$, such that there exist models (Definition \ref{models}) $S^{\ell}$, $\XX^{\ell}$, $X^{\ell}$ and $f^{\ell}:\XX^{\ell}\rightarrow S^{\ell}$ over $\ell$ of $S$, $\X$, $X$ and $f$ respectively (see \cite[II.5.10]{Kollar}), with the property that $f^\ell$ is fibrewise trivial with fiber $X^\ell$. Since $\trdeg(\ell/\mathbb{F})<\infty$, we find that $\trdeg(k/\ell)$ is infinite.\medskip

By Proposition \ref{prop:isomrepresentable}, $\mathcal{I}=\Isomm^{\XX, X\times_k S}_S$ and $\I^\ell = \Isomm^{\XX^\ell, X^\ell\times_k S^\ell}_{S^{\ell}}$ are schemes locally of finite type over $S$ and $S^\ell$ respectively. Moreover, we have $\mathcal{I}^\ell\otimes_\ell k\simeq \I$ and the structural morphisms $p:\mathcal{I}\rightarrow S$ and $p^\ell:\mathcal{I}^\ell\rightarrow S^\ell$ satisfy $p^\ell\times_{S^\ell}S = p$. By \ref{componentsofisom}, $\mathcal{I}^\ell=\bigsqcup_{i\in\N} \mathcal{I}^\ell_i$ with each $\mathcal{I}^\ell_i$ locally of finite type over $S^\ell$. Since $f$ is fibrewise trivial, $\mathcal{I}_s\simeq \Isomm^{\mathcal{X}_s,X}_k\simeq \Aut_k^X$ is nonempty at every $s \in S(k)$. Thus $S(k)\subseteq p(\mathcal{I}) = \bigcup_{i\in\N} p(\mathcal{I}_i)$ with $\I_i=\I^\ell_i\times_{S^\ell} S$. Again by \ref{componentsofisom}, each $p(\mathcal{I}_i)$ is constructible in $S$, and since each $\I_i$ is defined over $\ell$, it follows that each $p(\mathcal{I}_i)$ is defined over $\ell$ as well. Hence by Proposition \ref{countablecover} for some $j\in \N$ and $I:=\mathcal{I}_{j}$, the image $p(I)$ contains an open dense subset of $S$. Since $S$ is reduced, by \cite[Tag \href{https://stacks.math.columbia.edu/tag/052B}{\texttt{052B}}]{stacks} there exists $U\subseteq p(I)$ nonempty and open in $S$ such that $p:I_U\rightarrow U$ is faithfully flat. Then $(\id_{I_U}:I_U\to I_U)\in \Isom^{\XX_U,X\times_k U}_U (I_U)$ viewed as a morphism of $U$-schemes gives an isomorphism $\XX_{I_U}\simeq X\times_k I_U$ of $I_U$-schemes, proving that $\XX_U/U$ is fppf-locally trivial.
\medskip

Suppose furthermore that $I_s=\Aut_k^X$ is reduced. Then $I_s$ is smooth over $k$ by \cite[Proposition 1.28]{milne}, hence $p|_{I_U}:I_U\rightarrow U$ is smooth. By \cite[17.16.3]{EGA44}, there exists an étale surjective morphism $V\rightarrow U$ such that the restriction $I_V\rightarrow V$ admits a section $\sigma :V\rightarrow I_V$. Then $\sigma\in \Isomm^{\mathcal{X}_V,X\times_k V}_V(V)$ proving that $\mathcal{X}_V\simeq X\times_k V$. Therefore $f_U:\mathcal{X}_U\rightarrow U$ is isotrivial.
\end{proof}
\begin{remark}
	The above step is the only one in this proof where the assumption about $k$ being \textit{large} is used. Once the generic isotriviality holds for fibrewise trivial families over, $\overline{\mathbb F}$-varieties for $\mathbb F$ being a prime field, so does Theorem \ref{mainthm}, but the author does not know how to get rid of the assumption $\trdeg (k/\mathbb F)=\infty$.
\end{remark}

\section{Global vanishing of the Kodaira--Spencer map}
In this short section we review the construction of the Kodaira--Spencer map and prove that the Kodaira--Spencer map for fibrewise trivial families $f:\XX\to S$ satisfying assumptions as in Theorem \ref{mainthm} vanishes.
\subsection{Review of the Kodaira--Spencer map.} Let $f:\mathcal{X}\rightarrow S$ be a fibrewise trivial morphism of schemes with $S$ reduced. Since $f$ is smooth, recall that the cotangent sequence
\begin{equation}\label{42cotangent}
\begin{tikzcd}
	0\ar[r]&[-0.5 em]f^\ast\Omega^1_{S/k}\ar[r]&[-0.5 em]\Omega^1_{\mathcal{X}/k}\ar[r]&[-0.5 em] \Omega^1_{\mathcal{X}/S}\ar[r]&[-0.5 em] 0
\end{tikzcd}
\end{equation}
is exact \cite[17.2.3]{EGA44}. The cotangent sheaf $\Omega^1_{\mathcal{X}/S}$ is locally free, which implies that $\Ext^1(\Omega^1_{\mathcal{X/S}},\OO_\XX)$ is trivial, hence the dualization of \eqref{42cotangent}  
\begin{equation*}
\begin{tikzcd}
	0\ar[r]&\shT_{\mathcal{X}/S}\ar[r]&\shT_{\XX/k}\ar[r]& f^\ast\shT_{S/k}\ar[r]& 0
\end{tikzcd}
\end{equation*}
is also exact. After applying the left-exact functor $f_\ast$, we obtain a long exact sequence of $\OO_{S}$-modules
\begin{equation}\label{eq:long-exact-sequence-cotangent}
\begin{tikzcd}
	0\ar[r]&f_\ast\shT_{\mathcal{X}/S}\ar[r]&f_\ast\shT_{\XX/k}\ar[r]& f_\ast f^\ast\shT_{S/k}\ar{r}{\delta_{\XX/S}}& \RR^1f_\ast\shT_{\XX/S}\ar[r]&\cdots
\end{tikzcd}
\end{equation}
Denote the composition of $\delta_{\XX/S}$ with the adjunction unit \begin{equation}\shT_{S/k}\longrightarrow f_\ast f^\ast\shT_{S/k}\xrightarrow{\delta_{\XX/S}}\RR^1f_\ast\shT_{\XX/S}\end{equation} by $\KS_f$. Note that fibrewise triviality of $\XX/S$ implies $s\mapsto h^1(\XX_s,\shT_{\XX_s/k})$ is a constant function on $S(k)$. It follows from Grauert's Semicontinuity Theorem \cite[III.12.9]{hartshorne} that $\RR^1f_\ast\shT_{\XX/S}$ is locally free and that the natural morphism
\begin{equation*}
\RR^1f_\ast\shT_{\XX/S}(s)\longrightarrow H^1(\XX_s,\shT_{\XX_s/k})
\end{equation*}
 is an isomorphism. Therefore, the map $\KS_f$ specializes at a point $s\in S(k)$ under the above isomorphism to $\KS_f(s): \shT_{S/k}(s)\longrightarrow H^1(\XX_s,\shT_{\XX_s/k})$.
\begin{prop}\label{kodairavanishes}
	Assume $f\colon\XX\to S$ satisfies conditions of Theorem \ref{mainthm}. The Kodaira--Spencer morphism $\KS_f$ vanishes on $S$.
\end{prop}
\begin{proof}
	By Proposition \ref{genericisotriviality}, there exists an open dense subset $U\subset S$ such that $f_U$ is isotrivial. By \cite[Corollary 2.6.11]{sernesi}, $\KS_{f}(s)=0$ for all $s\in U(k)$. Since $\KS_f$ is a morphism of locally free sheaves and $S$ (hence $U$) is reduced, it follows that $\KS_f|_U\equiv 0$. Hence, the image of $\KS_{f}$ is a torsion subsheaf of $\RR^1f_\ast\shT_{\XX/S}.$ Since $\RR^1f_\ast\shT_{\XX/S}$ is locally free, it follows that the image must be trivial.	
\end{proof}
\begin{remark}
	In fact, the assumption that $\Aut_k^X$ is reduced is crucial for $\KS_f=0$. Later, in Section~\ref{KSnotzero}, we show the existence of fibrewise trivial families with nowhere vanishing Kodaira--Spencer maps. 
\end{remark}
\section{A variant of the valuative criterion of flatness}\label{valcritflat}
Our intermediate step in the proof of Theorems \ref{mainthm} is the reduction to the case when $S$ is a \emph{smooth} curve. As we show in Proposition \ref{isoimpliesflat}, isotriviality of a family $f\colon \XX\to S$ is equivalent to flatness of $\Isomm^{\XX,X\times S}_S$. It turns out that it is sufficient to check the latter condition only on smooth curves mapping to $S$. This is captured in the main result of this section, Proposition \ref{flatbccurves}. Its statement resembles the valuative criterion of flatness \cite[11.8.1]{EGA43}, but it appears that one cannot simply apply this result here, since not all valuation rings arise as stalks of structure sheaves of smooth $k$-curves. 
\begin{prop}\label{flatbccurves}
	Let $p\colon\I\rightarrow S$ be a finite type morphism with $S$ reduced and of finite type over $k$. Suppose furthermore that for every smooth integral curve $C$ over $k$ and $k$-morphism $C\rightarrow S$, the base change $p_C\colon\I_C\rightarrow C$ is flat. Then $p$ is flat.
\end{prop}

The fact itself seems to be known, being stated e.g. in McKernan's notes \cite[Lemma 4.10]{mckernan} but we couldn't find its proof. Hence, we present our own argument below.\medskip

The idea behind our proof is as follows: first, by blowing up singularities of $S$ we show that it is enough to restrict ourselves to morphisms with smooth targets. Then, working via induction on $\dim S$ we reduce the problem of flatness to the case when $S$ itself is a smooth integral curve.
	The main ingredient in the proof of Proposition \ref{flatbccurves} is the following fact.
\begin{prop}[{\cite[11.6.1]{EGA43}}]\label{unibranchflat}
	Let $A$ be a local, integral, geometrically unibranch \cite[Tag \href{https://stacks.math.columbia.edu/tag/06DT}{\texttt{06DT}}]{stacks} ring. Let $S=\Spec A$ and $f:X\rightarrow S$ be a finitely presented morphism. Let $A'$ be a local ring and $S'=\Spec A'$. Suppose $u^\sharp: A\rightarrow A'$ is a local injective homomorphism and let $u: S'\rightarrow S$ be the associated morphism of schemes. Let $x\in X$ with $s=f(x)$ closed in $S$. Suppose $x'\in X'=X\times_S S'$ is a point whose image in $X$ is equal to $x$ and whose image in $S'$ is the closed point $s'\in S'$. If $X'$ is $S'$-flat at $x'$, then $X$ is $S$-flat at $x$.
\end{prop}

Our proof of Proposition \ref{flatbccurves} consists of the following two lemmas.\begin{lemma}\label{smoothisenough}
	Suppose Proposition \ref{flatbccurves} holds when $S$ is smooth over $k$. Then Proposition \ref{flatbccurves} is true for an arbitrary reduced $S$ of finite type over $k$.
\end{lemma}
\begin{proof}
	Let $\{Z_i\}_{i=1}^m$ be the set of irreducible reduced components of $S$ and write $\I_i=\I\times_S Z_i$. By \cite[11.5.2]{EGA43}, it suffices to prove that each $\I_i$ is flat over $Z_i$. In other words, we may assume that $S$ is integral.

 Let $\nu:S'\rightarrow S$ be the normalization of $S$ in the function field $K(S)$. By Noether's theorem \cite[Corollary 13.13]{eisenbud} $\nu$ is a finite surjective morphism. By \cite[11.5.3]{EGA43}, it is enough to show that the base change $p': \I'=\I\times_{S}S'\rightarrow S'$ of $p$ is flat. Hence, we may additionally assume that $S$ is normal.
	
	Let $y\in \I(k)$ and $p(y)=s\in S$. Since $S$ is normal, $\OO_{S,s}$ is geometrically unibranch. Let $\rho: \widetilde S\rightarrow S$ be a normalization in $K(S)$ of a blow-up of $S$ at $s$. Then $\rho^{-1}(s)$ is a codimension $1$ subset in $\widetilde{S}$. Denote the smooth locus of $\widetilde S$ by $\widetilde S^{\circ}$. Since $\widetilde S$ is normal, $\widetilde S\setminus \widetilde S^{\circ}$ is a codimension $ \geq 2$ subset. Hence there exists $ \widetilde s \in \widetilde S^{\circ}$ satisfying $\rho(\widetilde s)=s$. By assumption $p_{\widetilde{S}^\circ}:\I_{\widetilde S^\circ}\rightarrow \widetilde S^{\circ}$ is flat. Since $\rho$ is birational, $\OO_{S,s}\rightarrow \OO_{\widetilde{S}^\circ, \widetilde s}$ is a local injective ring homomorphism. By Proposition \ref{unibranchflat}, since $\I_{\widetilde{S}^\circ}$ is $\widetilde{S}^\circ$-flat at $\widetilde y$ for any $\widetilde y$ lying above $y$ such that $p_{\widetilde S^\circ}(\widetilde y)=\widetilde s$, it follows that $\I$ is $S$-flat at $y$.
	\end{proof}
\begin{lemma}\label{curvesareenough}
	Suppose that Proposition \ref{flatbccurves} is true in case when $S$ is a smooth, reduced curve of finite type over $k$. Then Proposition \ref{flatbccurves} is true for every smooth, reduced $S$ of finite type over $k$.
\end{lemma}
\begin{proof}
	We may assume that $S$ is affine. Fix a point $y\in \I(k)$ and let $s=p(y)$. We argue by an induction on $\dim \OO_{S,s}$. The case $\dim \OO_{S,s}=1$ holds by our assumptions. Suppose the result holds for every $S$ of dimension smaller than $d$. Let $\dim S=d$ and let $\mathfrak{m}$ be the maximal ideal of $\OO_{S,s}$. Choose an element $f\in \mathfrak{m}\setminus \mathfrak{m}^2$. Suppose $f$ is a zero-divisor in $\OO_{\I,y}$. Let $g_1,g_2,\ldots, g_r\in\mathfrak{m}$ be such that together with $f$ they form a basis of the $k$-vector space $\mathfrak{m}/\mathfrak{m}^2$. Then $J=(g_1,g_2,\ldots, g_r)\subset \Gamma(S,\OO_{S,s})$ is an ideal satisfying $J+(f)=\mathfrak{m}$ (by Nakayama's Lemma) and $J\cap (f)=0\mod \mathfrak{m}^2$. Let $C$ be the curve in $S$ defined in a neighbourhood of $s$ by the ideal $J$. Then $\OO_{C,s}=\OO_{S,s}/J$ is a discrete valuation ring with uniformiser $f$ so $C$ is smooth at $s$. Let $\I_C=\I\times_S C$. Then $\OO_{\I_C,y}$ is not flat over $\OO_{C,s}$, a contradiction with our assumptions.
	
	Therefore the $\OO_{S,s}$-module homomorphism $\times f:\OO_{\I,y}\rightarrow \OO_{\I,y}$ is injective. Let $S'$ be the closed subscheme of $S$ cut out by $f=0$ and let $\I':=\I\times_S S'$. Since $f$ is not a zero divisor, $\dim S'<\dim S$ and by inductive assumption $\OO_{\I',y}$ is $\OO_{S',s}$-flat. Thus $\OO_{\I,y}$ is $\OO_{S,s}$-flat \cite[Corollary 6.9]{eisenbud}.
\end{proof}
\begin{proof}[Proof of \ref{flatbccurves}]
	Clearly the result holds in the case when $S$ is a smooth, integral curve. Now, combine Lemma \ref{smoothisenough} with Lemma \ref{curvesareenough}.
\end{proof}
\section{Proof of Theorem \ref{mainthm} when $\chr k=0$}\label{mainresult1}
In this section we prove Theorem \ref{mainthm}. Using global vanishing of Kodaira--Spencer map $\KS_f$, we show formal triviality of $f$ at every geometric point of $S$. We do so by  reducing the problem to the case when $S$ is a smooth curve. Then we prove the desired result by lifting $\id_{X}$ inductively to $S_n$-isomorphisms $\XX_{S_n}\rightarrow X\times S_n$ for all $n\geq 1$.
\medskip

\subsection{Formal triviality}\label{sef:formal-triviality} First, let us recall the notion of formal triviality and how formal triviality implies isotriviality in the case of projective families.
Let $f:\XX\rightarrow S$ be a morphism of schemes and let $s\in S(k)$. Denote the maximal ideal of $\OO_{S,s}$ with $\mathfrak m_s$ and write $S_n = \Spec \OO_{S,s}/\mathfrak m_s^{n+1}$. Take $\XX_n := \XX_{S_n}$, in particular $\XX_0=\XX_s$. Then, each $f_n:\XX_n\rightarrow S_n$ is a local deformation of $f_0:\XX_0\rightarrow S_0$. These mappings are compatible, i.e. the following diagram

\begin{equation*}
\begin{tikzcd}
	\cdots\ar[r]&\XX_{n-1}\ar{d}{f_{n-1}}\ar[r]& \XX_{n}\ar{d}{f_n}\ar[r]& \XX_{n+1}\ar{d}{f_{n+1}}\ar[r]& \cdots\\[1 em]\cdots\ar[r]& S_{n-1}\ar[r]& S_n\ar[r]& S_{n+1}\ar[r]&\cdots
\end{tikzcd}
\end{equation*}
commutes. Passing to the colimit, it gives an adic morphism of formal schemes $\mathfrak{X}\rightarrow\widehat{S}_s$, an element of $\widehat{\Def_{\XX_s}}(\widehat{\OO}_{S,s})$.
\begin{defn}\label{formaltriv}
	We say that $\XX/S$ is \textbf{formally trivial} at $s\in S(k)$ if for all $n$, $f_n:\XX_n\rightarrow S_n$ is a trivial deformation of $f_0$.
\end{defn}

Clearly, if a family $\XX/S$ is isotrivial at $s\in S(k)$, then it is formally trivial. The converse holds as well, provided the family in question is projective. Our proof of this fact is a variant of \cite[Corollary 2.6.10]{sernesi}.
\begin{prop}\label{formalisiso}
	Suppose $f$ is a flat projective family which is formally trivial at $s\in S(k)$. Then $f$ is isotrivial at $s$.
\end{prop}
\begin{proof}
The functor $F=\Isom^{\XX,X\times S}_S$ is locally of finite presentation. Since $f$ is formally trivial at $s\in S$, the formal completions of $\XX$ and $X\times S$ along $s$ are isomorphic as formal schemes over $\Spf\widehat{\OO}_{S,s}$. By Grothendieck's Existence Theorem (\cite[Corollary 8.4.7]{FGAexplained}, see also \cite[Theorem 2.5.11]{sernesi}), this isomorphism is induced by a unique isomorphism $\xi:\XX_{\widehat\OO_{S,s}}\rightarrow X\times \Spec \widehat\OO_{S,s}$ of $\widehat{\OO}_{S,s}$-schemes. Hence $\xi\in F(\widehat\OO_{S,s})$. By Artin's Approximation Theorem \cite[Theorem 1.12]{Artin}, there exists $\widetilde\xi\in F(\widetilde\OO_{S,s})$ with $\xi\equiv \widetilde\xi\pmod{\mathfrak{m}}$ for $\widetilde\xi\in F(\widetilde\OO_{S,s})$ the henselization of $\OO_{S,s}$. Let $\varphi:S'\rightarrow S$ be an \'etale neighborhood of $s$ such that $\widetilde\xi$ is the image of an element $\xi'\in F(S')$ along $F(\varphi)$. Then $\XX_{S'}$ and $X\times S'$ are isomorphic as $S'$-schemes. The result follows.
\end{proof}
As announced in Section \ref{valcritflat}, we show that under assumptions of Theorem \ref{mainthm}, isotriviality is equivalent to flatness of $\Isomm^{\X, X\times_k S}_S$.
\begin{prop}\label{isoimpliesflat}
	Let $f:\XX\rightarrow S$ be a smooth, projective, fibrewise trivial morphism with fiber $X$. Suppose $\Aut_k^X$ is reduced. Then $f$ is isotrivial if and only if\; $\Isomm^{\XX,X\times S}_S$ is flat over $S$.
\end{prop}
\begin{proof}
	By our assumptions $\Aut_k^X$ is $k$-smooth \cite[Proposition 1.28]{milne}, thus $\Isomm^{\XX,X\times S}_S$ is flat over $S$ if and only if it is smooth over $S$.\medskip
	
	If $\I=\Isomm^{\XX,X\times S}_S$ is smooth over $S$, then by \cite[17.16.3 (ii)]{EGA44} there exists an \'etale surjective morphism $S'\rightarrow S$ and an $S$-morphism $\sigma:S'\rightarrow \I$, proving that $\XX_{S'}\simeq X\times S'$.
	
	Conversely, suppose that $f$ is isotrivial. Let $\phi:S'\rightarrow S$ be an étale surjective morphism such that $\XX_{S'}\simeq X\times_S S'$ as $S'$-schemes. Then $\I_{S'}\simeq \Isomm^{\XX,X\times S}_{S}\times_S S'\simeq \Aut_k^X\times_k S'$ is obviously $S'$-flat. We have a Cartesian diagram
	\begin{displaymath}
		\begin{tikzcd}
			\I_{S'}\ar{r}{\psi}\ar{d}[swap]{p'} & \I \ar{d}{p}\\[1em] S'\ar{r}[swap]{\phi} & S
		\end{tikzcd}
	\end{displaymath}
	Clearly $\psi$ is flat and so is the composition $\phi\circ p'=p\circ \psi$. The result follows from \cite[Tag \href{https://stacks.math.columbia.edu/tag/02JZ}{\texttt{02JZ}}]{stacks} applied to $X=\I$, $Y=\I_{S'}$ and $\mathcal{G}=\OO_{\I}$.
\end{proof}

Now we are ready to prove our main result, under the assumption $\chr k=0$.
\begin{proof}[Proof of Theorem \ref{mainthm} in the case $\chr k=0$.]
By Proposition \ref{isoimpliesflat} it suffices to show that $\Isomm^{\XX,X\times S}_S$ is flat over $S$. By Lemma \ref{smoothisenough} and Lemma \ref{curvesareenough} it is enough to prove it under the assumption that $S$ is a smooth integral curve. Fix $s\in S(k)$. By Proposition \ref{formalisiso} and Proposition \ref{isoimpliesflat}, it suffices to show that $\XX/S$ is formally trivial at $s$. Write $\widehat\OO_{S,s}=k[[t]]$, let $S_n=\Spec k[t]/(t^{n+1})$, $\XX_n = \XX_{S_n}$ and $f_n:\XX_n\rightarrow S_n$.
	\\[-6 pt]
	
	We proceed by induction on $n$. When $n=0$, there is nothing to prove. Suppose that $n\geq 1$ and that we are given an isomorphism $\phi: \XX_{n-1}\rightarrow X\times S_{n-1}$ of $S_{n-1}$-schemes. Let $\mathfrak{U}=\{\mathcal{U}_i\}$ be an affine open cover of $\XX_n$, $U_i = \UU_i\times_{S_n}S_0$. Then $\{U_i\}$ is an affine open cover of $X_0$ and each $\UU_i\rightarrow S_n$ is a deformation of $U_i\rightarrow S_0$. Thus there exist isomorphisms $\phi_i:\UU_i\rightarrow U_i\times S_n$ \cite[Theorem 1.2.4]{sernesi}, each being a lift of the restriction of $\phi$ to $\UU_i\times_{S_n}S_{n-1}$ for every $i\in I$.
	\\[-6 pt]
	
	Denote $\mathcal{U}_{ij}=\mathcal{U}_i\cap\mathcal{U}_{j}$ and $U_{ij}=U_i\cap U_j$. Since all $X_m$ are separated, $U_{ij}$ and $\mathcal{U}_{ij}$ are also affine. Write $\phi_{ij}=\phi_i^{-1}\circ\phi_j$ and let $\psi_{ij}$ be the automorphism of $U_{ij}\times S_n$ making the following diagram of isomorphisms
\begin{displaymath}
\begin{tikzcd}
	\mathcal{U}_{ij}\ar{d}[swap]{\phi_{ij}}\ar{r}{\phi_j}&U_{ij}\times S_n\ar{d}{\psi_{ij}}\\[1.3 em] \mathcal{U}_{ij}\ar{r}[swap]{\phi_j}& U_{ij}\times S_n
\end{tikzcd}
\end{displaymath}
commute. An explicit calculation shows that $\psi_{ij} =\left(\phi_j|_{\mathcal{U}_{ij}}\right)\circ\left(\phi_i|_{\mathcal{U}_{ij}}\right)^{-1}$. By our inductive assumption each $\psi_{ij}$ descends to identity on $U_{ij}\times S_{n-1}$ under the base change $S_{n-1}\rightarrow S_n$. For $\pi: \Spec B\to \Spec A$, let $\pi^\sharp: A\to B$ denote the associated homomorphism of rings. We obtain $\psi_{ij}^\sharp=\id_{U_{ij}\times S_n}^\sharp+t^n\dd_{ij}$ for some $\dd_{ij}\in\Gamma(U_{ij},\shT_{U_{ij}\times S_n})$ and $\phi_{ij}^\sharp = \id_{\mathcal{U}_{ij}}^\sharp+\tau^n\phi^\ast_j\dd_{ij}$ for $\tau = f_n^\ast t$, the image of $t$ under the homomorphism $\Gamma(S_n,\OO_{S_n})\rightarrow \Gamma(\XX_n,\OO_{\XX_n})$ and $\phi^\ast_j\dd_{ij}:=\phi_j^{-1}\circ\dd_{ij}\circ\phi_j$. Also, observe that the equality $\phi_{ij}=\phi^{-1}_{ji}$ implies that $\dd_{ij}=-\dd_{ji}$ for all $i,j$.\medskip

Since $f$ is smooth, we have a short exact sequence
\begin{equation}\label{formaltrivialitytangent}
\begin{tikzcd}
	0\ar[r] &[-0.5 em] \shT_{X_n/S_n}\ar[r] &[-0.5 em]\shT_{X_n/k}\ar[r]& [-0.5 em]f_n^\ast\shT_{S_n/k}\ar[r] & [-0.5 em] 0.
\end{tikzcd}
 \end{equation}
The existence of isomorphisms $\phi_i$ implies that the above short exact sequence is locally split. For each $\UU_i$, consider the section $\sigma_i:\left.f_n^\ast\shT_{S_n/k}\right|_{\UU_i}\rightarrow \left.\shT_{X_n/k}\right|_{\UU_i}$ given by the composition \begin{equation}\label{section}
	f_n^\ast\shT_{S_n/k}|_{\mathcal{U}_i}\hookrightarrow \phi_i^\ast(\pr_{1i}^\ast\mathscr{T}_{U_i/k}\oplus \pr_{2i}^\ast\shT_{S_n/k})\simeq \mathscr{T}_{\mathcal{X}_n/k}|_{\mathcal{U}_i}.
\end{equation}
with $\pr_{1i}:U_i\times S_n\rightarrow U_i$ and $\pr_{2i}: U_i\times S_n\rightarrow S_n$ projections and the first arrow in \eqref{section} being the inclusion of the second direct summand. In particular, $f_n|_{\mathcal{U}_i}=\pr_{2i}\circ \phi_i$.
Consider the Čech complex associated to (\ref{formaltrivialitytangent}) and the covering $\mathfrak{U}$:
	\begin{displaymath}
	\begin{tikzcd}
	& 0\ar{d} & 0\ar{d} & 0\ar{d} &\\
		0\ar[r]&\bigoplus_i\Gamma(\mathcal{U}_i,\shT_{\XX_n/S_n})\ar[d]\ar[r] & \bigoplus_i\Gamma(\mathcal{U}_i,\shT_{\XX_n/k})\ar{d}{\delta} \ar[r]& \bigoplus_i\Gamma(\mathcal{U}_i,f_n^\ast \shT_{S_n/k})\ar[d]\ar[r]& 0\\
	 0\ar[r]&\bigoplus_{i,j}\Gamma(\mathcal{U}_{ij},\shT_{\XX_n/S_n})\ar[r]\ar{d} & \bigoplus_{i,j}\Gamma(\mathcal{U}_{ij},\shT_{\XX_n/k})\ar[r]\ar{d}& \bigoplus_{i,j}\Gamma(\mathcal{U}_{ij},f_n^\ast \shT_{S_n/k})\ar[r]\ar{d}& 0\\
	 & \vdots & \vdots & \vdots &\\
	\end{tikzcd}
\end{displaymath}
Let $\left[\left(\tau\frac{\dd}{\dd \tau}\right)|_{\mathcal{U}_i}\right]_i\in \bigoplus_i\Gamma(\mathcal{U}_i,f_n^\ast \shT_{S_n/k})$ be the product of restrictions of the global section $\tau\frac{\dd}{\dd \tau} = f_n^\ast\left(t\frac{\dd}{\dd t}\right)\in\Gamma(\XX_n,f_n^\ast\shT_{S_n/k})$ to each of $\mathcal{U}_i\in \mathfrak{U}$. By the above observations, the considered 0-cocycle lifts via chosen sections $\sigma_i$'s to the element $\left[\phi_i^\ast\pr_{2i}^\ast \left(t\frac{\dd}{\dd t}\right)\right]_i\in \bigoplus_i\Gamma(\mathcal{U}_i,\shT_{\XX_n/k})$. Write $t_i=\pr_{2i}^\ast t\in \Gamma(U_i\times S_n,\OO_{U_i\times S_n})$. We have $\OO_{U_i\times S_n}\simeq \OO_{U_i}[t_i]/(t_i^{n+1})$. It follows that $\pr_{2i}^\ast\left(\frac{\dd}{\dd t}\right) = \partial_i:=\frac{\partial}{\partial t_i}$, which is the $\OO_{U_i}$--derivation with respect to the variable $t_i$.\medskip

The map $\delta$ sends $\left[\phi_i^\ast\pr_{2i}^\ast \left(t\frac{\dd}{\dd t}\right)\right]_i$ to the element of $\bigoplus_{i,j}\Gamma(\mathcal{U}_{ij},\shT_{\XX_n/k})$ given by
\begin{displaymath}
	\delta\left[\phi_i^\ast\pr_{2i}^\ast \left(t\frac{\dd}{\dd t}\right)_i\right]_{ij} = \left.\phi_i^\ast\pr_{2i}^\ast \left(t\frac{\dd}{\dd t}\right)\right|_{\UU_{ij}} - \left.\phi_j^\ast\pr_{2j}^\ast \left(t\frac{\dd}{\dd t}\right)\right|_{\UU_{ij}} = \left.\phi_i^\ast \left(t_i\partial_i\right)\right|_{\UU_{ij}} - \left.\phi_j^\ast\left(t_j\partial_j\right)\right|_{\UU_{ij}}.
\end{displaymath}
Now, observe that
\begin{align*}
\begin{aligned}
	\left.\phi_i^\ast \left(t_i\partial_i\right)\right|_{\mathcal{U}_{ij}} - \left.\phi_j^\ast\left(t_j\partial_j\right)\right|_{\mathcal{U}_{ij}} &= \left.\phi_i^\ast\left(t_i\partial_i - (\phi_j\phi_i^{-1})^\ast \left(t_j\partial_j \right)\right)\right|_{\mathcal{U}_{ij}}\\& = \phi^\ast_i\left(t_i\partial_i - \psi_{ij}^\ast \left(t_j\partial_j\right)\right).
	\end{aligned}
\end{align*}
Therefore, we obtain
\begin{align*}\label{pullbackused}
\begin{aligned}
	\psi_{ij}^\ast\left( t_j\frac{\partial}{\partial t_j}\right) &= \psi_{ij}^\sharp\circ\left( t_j\frac{\partial}{\partial t_j}\right)\circ \psi_{ji}^\sharp\\& = \psi_{ij}^\sharp \left(t_j\frac{\partial}{\partial t_j}\left(\id-t_j^n\dd_{ij}\right)\right) = \psi^\sharp_{ij}\left(t_j\frac{\partial}{\partial t_j}-nt_j^{n}\dd_{ij}-\underbrace{t_j^{n+1}\dd_{ij}\circ\frac{\partial}{\partial t_j}}_{=0}\right) \\&= t_j\frac{\partial}{\partial t_j} - nt_j^n\dd_{ij} = t_j\frac{\partial}{\partial t_j}+nt_j^{n}\dd_{ji}.
	\end{aligned}
\end{align*}
Since $f_n|_{\mathcal{U}_{ij}}=\pr_{2i}\circ \left(\phi_i|_{\mathcal{U}_{ij}}\right) = \pr_{2j}\circ \left(\phi_j|_{\mathcal{U}_{ij}}\right)$, we get
\begin{displaymath}
	\left.\phi_i^\ast\left(t_i\frac{\partial}{\partial t_i}\right)\right|_{\mathcal{U}_{ij}} - \left.\phi_j^\ast\left(t_j\frac{\partial}{\partial t_j}\right)\right|_{\mathcal{U}_{ij}} = \phi_i^\ast\left[\left(n t_j^n\dd_{ji}\right)|_{\mathcal{U}_{ij}}\right] = n\tau^n\phi_i^\ast\dd_{ji}.
\end{displaymath}
As a result, we obtain a $1$-cocycle $\left(n\tau^n\phi_i^\ast\dd_{ij}\right)_{i,j}\in \bigoplus_{i,j}\Gamma(\mathcal{U}_{ij},\shT_{\XX_n/S_n})$. It is the image of $t\frac{\dd}{\dd t}\in H^0(\XX_n,f_n^\ast\shT_{S_n/k})$ under the map $H^0(\XX_n,f_n^\ast\shT_{S_n/k})\rightarrow H^1(\XX_n,\shT_{\XX_n/S_n})$ arising from the long exact sequence. This is the morphism of $\Gamma(S_n,\OO_{S_n})$-modules that makes the diagram
\begin{equation*}
\begin{tikzcd}
	\widetilde{H^0(\XX_n,f_n^\ast\shT_{S_n/k})}\ar{d}{\simeq}\ar[r]& \widetilde{H^1(\XX_n,\shT_{\XX_n/S_n})}\ar{d}{\simeq}\\ (f_n)_\ast f_n^\ast\shT_{S_n/k}\ar{r}{\delta_{\XX/S}} & \RR^1f_\ast\shT_{\XX_n/S_n}
\end{tikzcd}
\end{equation*}
commute. Furthermore, $\tau\frac{\dd}{\dd \tau}\in\Gamma(\XX_n,f_n^\ast\shT_{S_n/k})=\Gamma(S_n,(f_n)_\ast f_n^\ast\shT_{S_n/k})$ is clearly in the image of $\shT_{S_n/k}\rightarrow (f_n)_\ast f_n^\ast\shT_{S_n/k}$. Therefore, since $S$ is smooth, $\delta_{\XX_n/S_n}=\delta_{\XX/S}|_{S_n}$ and by Proposition \ref{kodairavanishes}, $\delta_{\XX_n/S_n}\left(t\frac{\dd}{\dd t}\right)=0$. As a result, the cohomology class of $\left(n\tau^n\phi_i^\ast\dd_{ji}\right)_{i,j}\in \bigoplus_{i,j}\Gamma(\mathcal{U}_{ij},\shT_{\XX_n/S_n})$ is trivial. Thus, since $\ch k=0$ and $n$ is invertible, the class $[\phi_i^\ast\dd_{ij}]_{i,j}=-[\phi_i^\ast\dd_{ji}]_{i,j}$ is also a 1-coboundary. However, by the construction of derivations $\dd_{ij}$, this implies that $f_n$ is trivial as a deformation of $f_{n-1}$. In other words $\phi$ lifts to a morphism $\widetilde{\phi}:\XX_n\rightarrow X\times S_n$ which completes the inductive step and proves that $\XX/S$ is formally trivial at $s$.
\end{proof}
\begin{remark}\label{useofchar0}
	Observe that the last part of the proof of Theorem \ref{mainthm} relies on the assumption that $\chr k=0$.
\end{remark}
\section{Proof of Theorem \ref{mainthm} when $\chr k=p>0$}\label{mainresult2}
In this section, we prove Theorem \ref{mainthm} in the case when $\chr k=p>0$. Our main idea is the same as in the proof of characteristic zero case, i.e. we deduce the result from vanishing of $\KS_f: \shT_{S}\to \RR^1f_\ast \shT_{\X/S}$. However, in order to prove formal triviality of $\XX/S$, we observe that \'etale locally, our family can be descended along the absolute Frobenius $F_S: S\to S^{(1)}$. This is inspired by the following result appearing in the work of Ogus.
\begin{lemma}[{\cite[Lemma 3.5]{ogus}}]\label{frobeniusdescend}
	Suppose $f:\XX\to S=\Spec k[[t_1,t_2,\ldots, t_r]]$ is smooth and proper, let $F_S:S\to S^{(1)}$ be the Frobenius map. Suppose that the Kodaira--Spencer class $\xi\in H^1(\XX,\shT_{\XX/S})\otimes \Omega^1_{S/k}$ is trivial. Then $\XX/S$ can be descended to $S^{(1)}$, meaning that there exists $\XX^{(1)}/S^{(1)}$ and a Cartesian square
	\begin{equation*}
		\begin{tikzcd}
			\XX\ar{r}\ar{d}[swap]{f} & \XX^{(1)}\ar{d}{f^{(1)}}\\[1 em] S\ar{r}[swap]{F_S} & S^{(1)}.
		\end{tikzcd}
	\end{equation*}
\end{lemma}
Yet, the above cannot be applied directly to $\XX/S$ with $S$ (locally) of finite type over $k$. For the purpose of isotriviality, however, it is enough to descend along Frobenius only \emph{\'etale-locally} on $S$. Also, in order to apply this technique inductively, we need to ensure that there exists an \'etale neighborhood of any given point $s\in S(k)$, such that we can force our local Frobenius descent of $\X$ to be smooth projective and fibrewise trivial. Combining Lemma \ref{frobeniusdescend} with Popescu's Theorem, we propose the following workaround of this issue.
\begin{prop}\label{ogus-refinement}
	Let $S$ be a smooth curve, $s\in S(k)$ and $f:\XX\to S$ be a smooth projective morphism. Suppose the Kodaira--Spencer map $\KS_f:\shT_{S}\to \RR^1f_\ast \shT_{\XX/S}$ vanishes. Then there exists an \'etale neighborhood $s\to U\to S$ such that $f_U:\XX_U\to U$ descends to a smooth, projective, fibrewise trivial family $f^{(1)}:\XX^{(1)}\to U^{(1)}$ along the Frobenius $F_U:U\to U^{(1)}$.
\end{prop}
\begin{proof}
	We may assume that $S$ is affine, say $S=\Spec A$. By Popescu's Theorem \cite[Tag \href{https://stacks.math.columbia.edu/tag/07GC}{\texttt{07GC}}]{stacks} the ring homomorphism
	\begin{equation*}
		\Gamma(S,\OO_{S})=:A\to \widehat\OO_{S,s}
	\end{equation*}
	is a colimit of a filtered diagram $\mathcal D$ of smooth $A$-algebras. By \cite[8.9.1.(iii)]{EGA43} there exists $R_0\in \mathcal D$ and a model $\widehat f^{(1)}_0: \mathscr{X}_0\to \Spec R_0^{(1)}$ of $\widehat{f}^{(1)}$. Then $\widehat\OO_{S,s}$ is a filtered colimit of $R_0$-algebras $R_\alpha$ which are smooth $A$-algebras. Let \[\mathscr{X}_\alpha:=\mathscr{X}_0\otimes_{R_0^{(1)}}R_\alpha^{(1)}\xrightarrow{\widehat f_\alpha^{(1)}}\Spec R_\alpha^{(1)}.\] Then, by \cite[8.10.5.(xiii)]{EGA43} and \cite[17.7.8.(ii)]{EGA44} there exists $\alpha$ such that $\widehat f_\alpha^{(1)}\colon\mathscr X_\alpha\to \Spec R_\alpha^{(1)}$ is smooth and projective with $\widehat f_\alpha^{(1)}\otimes_{R_\alpha^{(1)}}\widehat \OO_{S,s}^{(1)}=\widehat f^{(1)}$. Denote $R=R_\alpha$ and $\mathscr X^{(R)}=\mathscr{X}_\alpha$.\medskip
	
	Next, let $\mathcal D_{R/}$ be the coslice category corresponding to $R\in \mathcal D$. Since $\mathcal D$ is filtered, we have $\colim \mathcal D = \colim \mathcal D_{R/}$. Consider the functor $\I = \Isom_{\XX_R, F^\ast_{R}\mathscr X^{(R)}}$ which is locally finitely presented. Recall that we are given an isomorphism	
	\begin{equation*}
		\iota=\left(\XX_R\otimes_R \widehat \OO_{S,s}\xrightarrow{\sim} \widehat \XX\xrightarrow{\sim}F_{\widehat S}^\ast \mathscr X\xrightarrow{\sim} (F_{R}^\ast \mathscr{X}^{(R)})\otimes_R \widehat \OO_{S,s}\right)\in \I(\widehat \OO_{S,s})=\I(\colim \mathcal D_{R/}).
	\end{equation*}
	Since $\I$ is locally of finite presentation, we have $\I(\colim \mathcal D_{R/}) = \colim \I(\mathcal D_{R/})$, so $\iota$ is the image of an element $\iota_C\in \I(C)$ for some $C\in \mathcal D_{R/}$. In other words, after pulling back $\mathscr X^{(R)}$ along $R\to C$, we obtain an isomorphism
	\begin{equation*}
		\iota_C: \XX_C\xrightarrow{\sim} \XX_R\otimes_R C\to (F_R^\ast\mathscr{X}^{(R)})\otimes_R C\xrightarrow{\sim}F_C^\ast(\mathscr{X}^{(R)}\otimes_{R^{(1)}} C^{(1)})
	\end{equation*}
	satisfying $\iota_C\otimes_C\widehat\OO_{S,s}=\iota$. Denote $\mathscr X^{(C)}:=\mathscr{X}^{(R)}\otimes_R C$. Finally, take $U\to S$ to be an \'etale neighborhood of $s$ which admits a section $U\to \Spec C$ over $S$. Then $\X^{(1)}:=\mathscr{X}^{(C)}\times_{C^{(1)}} U^{(1)}$ and $\iota_U:=\iota_C\otimes_C U: \XX_U\xrightarrow{\sim} F^\ast_U\XX^{(1)}$ satisfy the desired properties.
\end{proof}
Now we can finish the proof of Theorem \ref{mainthm}.
\begin{proof}[Proof of Theorem \ref{mainthm} in the case $\chr k=p>0$]
	It suffices to show that $\XX/S$ is formally trivial at every point $s\in S(k)$. Fix such a point $s$, let $\mathfrak m_s$ be the maximal ideal in $\OO_{S,s}$, and let $S_n=\Spec \OO_{S,s}/\mathfrak m_s^{n+1}$. By Theorem \ref{genericisotriviality}, the family $\XX/S$ is generically isotrivial, hence the Kodaira--Spencer map $\KS_f:\shT_{S}\to \RR^1f_\ast\shT_{\XX/S}$ vanishes (Proposition \ref{kodairavanishes}). By Proposition \ref{ogus-refinement} we find that there exists an \'etale neighborhood $U\to S$ of $s$ such that $\XX_U$ descends to some smooth, projective and fibrewise trivial $f^{(1)}:\XX^{(1)}/U^{(1)}$ with fibre $X$ along Frobenius $F_U: U\to U^{(1)}$. Repeating the above argument, we find a sequence of \'etale neighborhoods
	\begin{equation*}
		\cdots \to U_{m+1}\to U_{m}\to U_{m-1}\to \cdots \to U_1\to S
	\end{equation*}
	of $s$ such that $\XX_{U_m}$ descends to some $f^{(m)}:\XX^{(m)}\to U_k^{(m)}$ along $m^{\rm th}$-power of Frobenius $F_{U_m}^m\colon U_m\to U_m^{(m)}$. This in particular shows that $\XX\times_S S_{p^m-1} \simeq X\times_k S_{p^m-1}$ for all $m\in \N$. Therefore $\XX$ is formally trivial at $s$.
\end{proof}
In particular, we obtain the following result.
\begin{cor}\label{isowhendefisprorep}
 Suppose $k=\overline k$, $\trdeg(k/\mathbb F)=\infty$ and $f\colon \XX\to S$ is smooth, projective, fibrewise trivial family with fibre $X$. Suppose furthermore that $\Def_X$ is pro-representable. Then $\XX/S$ is isotrivial.	
\end{cor}
This follows easily from the following lemma.
\begin{lemma}\label{lemma:isowhendefisprorep}
	Suppose $X$ is smooth projective $k$-scheme such that the functor $\Def_X:\Art_k \to \Sets$ is pro-representable. Then the group scheme $\Aut_k^X$ is reduced.
\end{lemma}
\begin{proof}
	Adapting the terminology of \cite{sernesi}, a morphism $\phi:A'\to A$ in $\Art_k$ is called a \textbf{small extension} if $\ker\phi = (t)$ for some $t\in A'$. Clearly for every $A\in \Art_k$, the quotient morphism $A\to k$ can be written as a finite composition of small extensions.  Since $\Def_X$ is prorepresentable, by \cite[Theorem 2.6.1]{sernesi} we know that given any small extension $A'\to A$, the group homomorphism
	\begin{equation*}
		\Aut_{A'}(X\times_k A')\to \Aut_{A}(X\times_k A)
	\end{equation*}
	is surjective. However, for any $A\in \Art_k$, we have \[\Aut_A(X\times_k A)=\Autt^X_k(\Spec A)=\Hom_k(\Spec A, \Aut^X_k),\]
	thus, for every small extension $\phi\colon A'\to A$, the induced map
\begin{equation}\label{eq:surjectiveonauthom}
	\phi^\ast: \Hom_k(\Spec A', \Aut_{k}^X)\to \Hom(\Spec A, \Aut_k^X)
\end{equation}
is surjective. As a result \eqref{eq:surjectiveonauthom} holds for arbitrary $A'\in \Art_k$ and $A=k$. Hence, denoting with $m_{A'}$ the maximal ideal of $A'$, every morphism $\chi:\Spec A'\to \Aut^X_k$ factors
\begin{equation*}
	\begin{tikzcd}
		\Spec k\ar[hook]{r}\ar[bend left]{rr}[swap]{\phi^\ast(\chi)}\ar[equal]{rd}& \Spec A'\ar{r}{\chi}\ar{d} & \Aut_k^X\\[1 em] &\Spec A/\mathfrak m_{A'}=\Spec k\ar[dashed]{ru}&
	\end{tikzcd}
\end{equation*}
and thereby proves that $\Aut^X_k$ is reduced.
\end{proof}
\begin{proof}[Proof of Corollary \ref{isowhendefisprorep}]
	It is a formal consequence of Corollary \ref{lemma:isowhendefisprorep} and Theorem \ref{mainthm}.
\end{proof}
\section{Fibrewise trivial family with non-vanishing Kodaira-Spencer map.}\label{KSnotzero}
In this section, we employ quotients of trivial families by a group action in order to construct fibrewise trivial families $f\colon\XX\to S$ over a smooth curve the Kodaira--Spencer map $\KS_f(s): \shT_{S}(s)\to H^1(\XX_s,\shT_{\XX_s})$ is nonzero at every $s\in S$. In particular, such a family is not isotrivial, thus showing that the assumption that $\Aut^X_k$ is reduced in Theorem \ref{mainthm} is necessary.

\subsection{Quotients of trivial families by diagonal group actions}\label{diagonalquotients} 

Let $S$ and $X$ be two $k$-schemes such that $X$ is smooth and projective. Then the group scheme $\Aut^X_k$ exists. Let $G$ be a group scheme with a group scheme homomorphism $\iota: G\to \Aut^X_k$, which in turn (by the exponential law \eqref{eq:explaw}) induces a left group action $\mu: G\times X\to X$. Given $g:T\to G$, we write $\iota(g): T\times X\to T\times X$ the $T$-automorphism associated to $\iota\circ g: T\to \Aut^X_k$. Then, by the adjunction \eqref{eq:explaw} we have that the square
\begin{equation*}
	\begin{tikzcd}
		T\times X\ar{d}[swap]{g\times \id_X}\ar{r}{\iota(g)} & T\times X\ar{d}{\pr_2}\\[1 em] G\times X\ar{r}[swap]{\mu} & X
	\end{tikzcd}
\end{equation*}
commutes. Let $S'\to S$ be an fppf $G$-torsor, $m: S'\times G\to S'$ be the associated right $G$-action on $S'$. Then, we equip $S'\times X$ with \textbf{diagonal (right) action} of $G$: $((s',x),g)\mapsto (sg, \iota(g)^{-1}x)$.
Suppose furthermore that the fppf quotient $(S'\times X)/G$ is representable by a scheme $\XX$. Observe that we have the following commutative square
\begin{equation}\label{eq:diagonalaction}
	\begin{tikzcd}
		S'\times G\times X\ar{r}{\id_{S'}\times \mu}\ar{d}[swap]{m\times \id_X} & S'\times X\ar{d}{\pi}\\ [1 em] S\times X\ar{r}[swap]{\pi} & \XX.
	\end{tikzcd}
\end{equation}
Denote by $f:\XX\to S$ the induced map on quotients so that the following square
\begin{equation*}
	\begin{tikzcd}
		S'\times X\ar{r}{\pi}\ar{d}[swap]{\pr_1} & \XX\ar{d}{f}\\[1 em] S'\ar{r}[swap]{q} &S
	\end{tikzcd}
\end{equation*}
commutes with the horizontal arrows being quotient maps. Then, denoting $q^\ast \XX$ the pullback of $\XX$ along $q$, we find that the induced canonical morphism $\phi: S'\times_k X\to q^\ast \XX$ is an ($G$-equivariant) isomorphism, i.e. the above square is cartesian.

Similarly, it is clear that the morphism $q^\ast f: q^\ast\XX\to S'$ is $G$-equivariant as well. The isomorphism $\phi$ induces a morphism $[\phi]:S'\to \I:=\Isomm^{S\times X,\XX}_S$ of $S$-schemes. We know that $S'$ is equipped with right $G$-action and at the same time, $\mathcal H:=\Hom_S^{S\times X, \XX}$ is equipped with right $\Aut_k^X$-action $\nu: \mathcal H\times_k\Aut^X_k\to \mathcal H$ defined as follows: for $\sigma\in \Aut_k^X(T)$ and $(\alpha: T\times X \to T\times_S \XX)\in \I(T)$ we take
\begin{equation}\label{eq:group-action-isom}
	\nu(\alpha,\sigma): T\times X\xrightarrow{\sigma} T\times X\xrightarrow{\alpha} T\times_S \XX .
\end{equation}
This action obviously restricts to $\I=\Isom^{S\times X,\XX}_S$ which is an open subscheme of $\mathcal H$. This $\Aut^X_k$-action on $\I$ makes $\I$ into $\Aut^X_k$-torsor, as it is trivialised by the fppf cover $S'\to S$. With a slight abuse of notation, we will denote this action by $\nu$ as well. We claim that the two group actions on $S'$ and $\I$ are compatible with respect to $[\phi]:S'\to \I$ and $\iota$.\medskip

\begin{prop}\label{equivariantmap}
	Under the assumptions made in \ref{diagonalquotients}, the morphism $[\phi]:S'\to \I$ is equivariant with respect to $\iota:G\to \Aut_k^X$, i.e. the following square
\begin{equation}\label{eq:desired-comm-sq}
	\begin{tikzcd}
		S'\times G\ar{r}{([\phi],\iota)}\ar{d}{m} &[1 em] \I\times_k\Aut_k^X\ar{d}{\nu}\\[2 em] S'\ar{r}{[\phi]} & \I
	\end{tikzcd}
\end{equation}
is commutative.
\end{prop}
\begin{proof}Since $\I\subset \mathcal H$ is a monomorphism, to prove the desired claim it suffices to show that the composition of the two squares below
\begin{equation*}
	\begin{tikzcd}
		S'\times G\ar{r}{([\phi],\iota)}\ar{d}[swap]{m} &[1 em] \I\times_k\Aut_k^X\ar{d}{\nu}\ar[hook]{r} & \mathcal H\times_k \Aut^k_X\ar{d}{\nu} \\[2 em] S'\ar{r}[swap]{[\phi]} & \I \ar[hook]{r}& \mathcal H
	\end{tikzcd}
\end{equation*}
is commutative. This in turn is equivalent with proving that
\begin{equation*}
	\begin{tikzcd}
		\Hom(T,S')\times \Hom(T,G)\ar{r}{([\phi],\iota)}\ar{d}{m} &[1 em] \Hom(T,\mathcal H)\times \Hom(T,\Aut_k^X)\ar{d}{\nu} \\[2 em] \Hom(T,S')\ar{r}{[\phi]} & \Hom(T,\mathcal H)
		\end{tikzcd}
\end{equation*}
commutes, which, due to the adjunction \eqref{eq:explaw}, is equivalent with proving the commutativity of the square
\begin{equation}\label{eq:commutativity2}
	\begin{tikzcd}
		\Hom(T,S')\times \Hom(T,G)\ar{r}{([\phi],\iota)}\ar{d}{m} &[1 em] \Hom(T\times X,\XX)\times \Hom(T,\Aut_k^X)\ar{d}{\nu} \\[2 em] \Hom(T,S')\ar{r}{[\phi]} & \Hom(T\times X,\XX).
	\end{tikzcd}
\end{equation}
Here, the upper horizontal arrow sends $(s',g)\in \Hom(T,S')\times \Hom(T,G)$ to the pair
\begin{equation*}
	\begin{tikzcd}
		T\times X\ar{r}{s'\times \id_X} & S'\times X\ar{r}{\phi}[swap]{\sim} & S'\times_S \XX\ar{r}{\pr_2}& \XX, & \iota\circ g: T\ar{r} & \Aut^X_k,
	\end{tikzcd}
\end{equation*}
and the action (right vertical arrow) sends $(\Phi: T\times X\to \XX, \sigma\in \Aut^k_X(T))$ to the composition
\begin{equation*}
	\begin{tikzcd}
		T\times X \ar{r}{\sigma} & T\times X\ar{r}{\Phi} & \XX.
	\end{tikzcd}
\end{equation*}
Therefore, the clockwise composition sends $(s',g)$ to the composition
\begin{equation*}
	\begin{tikzcd}
		T\times X \ar{r}{\iota(g)} & T\times X\ar{r}{(s'\times \id_X)} & S'\times X\ar{r}{\phi} &S'\times_S \XX \ar{r}{\pr_2} & \XX.
	\end{tikzcd}
\end{equation*}
Meanwhile, the counter-clockwise composition sends $(s',g)$ to
\begin{equation*}
	\begin{tikzcd}
		T\times X\ar{r}{(s', g)\times \id_X} &[1.5 em] S'\times G\times X \ar{r}{m\times\id_X} &[1 em] S'\times X\ar{r}{\phi} & S'\times_S \XX\ar{r}{\pr_2} & \XX.
	\end{tikzcd}
\end{equation*}
These two compositions are equal because $\pi: S'\times X\xrightarrow{\phi} S'\times_S \XX \xrightarrow{\pr_2}  \XX$ is $G$-invariant, i.e. for any $(s',g)\in S'(T)\times G(T)$ the diagram
\begin{equation*}
	\begin{tikzcd}
		T\times X\ar{r}{(s',g)\times \id_X}\ar[bend left]{rr}{(s',\pr_2\circ\iota(g))}&[2 em] S'\times G\times X \ar{r}{\id_{S'}\times \mu}\ar{d}{m\times\id_X} &[1 em] S'\times X\ar{d}{\pi}\\[1 em] & S'\times X\ar{r}{\pi} & \XX
	\end{tikzcd}
\end{equation*}
commutes (see \eqref{eq:diagonalaction}). This proves commutativity of \eqref{eq:commutativity2} for all $T\in \Sch_S$ which, by Yoneda's Lemma, implies desired commutativity of \eqref{eq:desired-comm-sq} and completes the proof.
\end{proof}

Recall that given a group scheme $G$ and $S$ equipped with a Grothendieck topology $\tau$, isomorphism classes of $G$-torsors with respect to $\tau$ on $S$ are classified by classes in ${\rm H}^1_\tau (S,G)$ \cite[Tag \href{https://stacks.math.columbia.edu/tag/03AJ}{\texttt{03AJ}}]{stacks}.

A formal consequence of this correspondence and Proposition \ref{equivariantmap} is the following statement.
\begin{cor}
	With notation as above, the image of $G$-torsor $S'/S$ under the map \[{\rm H}^1_{\rm fppf}(\iota):{\rm H}^1_{{\rm fppf}}(S,G)\to {\rm H}^1_{\rm fppf}(S,\Aut^X_k)\] coincides with $\I=\Isomm^{\XX,X\times_k S}_S$.
\end{cor}
\subsection{Fibrewise trivial families with nowhere vanishing Kodaira--Spencer map}\label{fib-triv-nonzero-KS} 
Suppose that $\chr k=p>0$ and let $X\in \Sch_k$ be a smooth projective scheme such that the connected component of $\Aut^X_k$ containing $1$ is isomorphic to $\mu_p=:G$. For instance, one can set $\chr k=2$ and choose $X$ to be a generic supersingular Enriques surface \cite[Example 5.6]{martin2020infinitesimal}. Take $S=\mathbb G_m$, $S'=\mathbb{G}_m$, and let $S'\to S$ be the $p^{\rm th}$ power map, which makes $S'$ into a nontrivial $\mu_p$-torsor over $S$. Then the quotient $(S'\times_k X)/G$ by the diagonal action is representable by a scheme $\XX$ \cite[Section 4]{tziolas}. Again, let $f:\XX\to S$ be as in \ref{diagonalquotients}.
\begin{prop}
	Consider the setup as in \ref{fib-triv-nonzero-KS}. Then the Kodaira--Spencer map of $f\colon\XX\to S$ is nowhere vanishing. In particular, $f$ is not isotrivial.
\end{prop}
\begin{proof}
Pick $s\in S(k)$. To show that $\KS_f$ does not vanish at $s$, it suffices to show that $\KS_{(\times s)^\ast f}$ does not vanish at $s=1$ for $(\times s)^\ast f: (\times s)^\ast\XX\to S$ being the pullback of $f$ along $S\xrightarrow{\times s} S$, so we may simply assume that $s=1$.

Let $v:\Spec k[\varepsilon]=:D\to \mathbb G_m=\Spec k[t]$ be given by $t\mapsto 1+\epsilon$. It is enough to show that $v^\ast\XX\to D$ is a non-trivial deformation of $X/k$. For this to hold, it suffices to prove that $\Isomm^{v^\ast\XX, D\times_k X}_D=v^\ast\I \to D$ does not admit a section, where $\I=\Isomm^{\XX,S\times X}_S\to S$. If it did have a section then $v^\ast\I \simeq \Aut^{X}_k\times_k D$ and its cohomology class $[v^\ast \I]={\rm H}_{\rm fppf}^1(\iota)([F])$ would be zero in ${\rm H}^1_{\rm fppf}(D,\Aut^X_k)$. However we have the commutative diagram
\begin{equation}\label{eq:h1fppf}
	\begin{tikzcd}
		H^1_{\rm fppf}(S,G)\ar[hook]{r}{H^1_{\rm fppf}(\iota)}\ar{d}{v^\ast} &[1 em] H^1_{\rm fppf}(S,\Aut^X_k)\ar{d}{v^\ast}\\ [1 em]
		H^1_{\rm fppf}(D,G)\ar[hook]{r}{H^1_{\rm fppf}(\iota)} & H^1_{\rm fppf}(D,\Aut^X_k)
	\end{tikzcd}
\end{equation}
Clearly the image $v^\ast [S']\in H^1_{\rm fppf}(D,G)$ corresponds to the morphism $\Spec k[\epsilon']/((\epsilon')^{2p})\to \Spec k[\epsilon]/(\epsilon^2)$ given by $\epsilon\mapsto (\epsilon')^p$ which is a nontrivial $G$-torsor. It remains to show that the bottom map in \eqref{eq:h1fppf} has trivial kernel --- note that since $G=(\Aut^X_k)^0$, the composition $(\Aut^X_k)_{\rm red}\subset \Aut^X_k\to \Aut^X_k/G$ is an isomorphism, hence the epimorphism in the exact sequence
\begin{equation}\label{eq:shortexactsequence-groupsection}
	0\to G\xrightarrow{\iota} \Aut^X_k\xrightarrow{q} \Aut^X_k/G\to 0
\end{equation}
admits a section. For every $T\in \Sch_k$, we have the induced exact sequence of pointed sets \cite[Ch. III, Prop 3.2.2]{MR0344253}
\begin{equation}\label{eq:longexactsequence-torsors}
	0\to H^0(T,G)\to H^0(T,\Aut^X_k)\to H^0(T,\Aut^X_k/G)\xrightarrow{d} H^1_{\rm fppf}(T,G)\to H^1_{\rm fppf}(T,\Aut^X_k).
\end{equation}
Suppose that $T'/T\in H^1_{\rm fppf}(T,G)$ induces the trivial $\Aut^X_k$-torsor. By \eqref{eq:longexactsequence-torsors} there exists a section $s:T\to \Aut^X_k/G$ such that $d(s)=[T']$. By \eqref{eq:shortexactsequence-groupsection} $s$ lifts to $\widetilde s:T\to \Aut^X_k$, so that $q\widetilde s = s$. Therefore $[T']=d(q_\ast\widetilde s)=\ast$ by exactness of \eqref{eq:longexactsequence-torsors}, proving the triviality of $T'$. Apply this to $T=D$.
\end{proof}
\begin{remark}
	We briefly remark on a slightly different example of a fibrewise trivial family which is not isotrivial. Let $X$ be a surface with ample canonical divisor $K_X$, and such that $H^0(X,\shT_X)\neq 0$. Then the scheme of automorphisms $\Aut^X_k$ is finite and non-reduced. Suppose furthermore that $\Aut^X_k$ contains $\mu_p=\Spec k[t]/(t^p-1)$ as a subgroup.  Here, one can take for instance $p=5$ and a hypersurface in $\PP^3$ of degree $5$ (which is automatically of general type), invariant under the Godeaux action on $\PP^3$ \cite[1]{lang-godeaux}, and on which the restriction of $\mu_5$-action is free. Then for the $\mu_p$-torsor $S'\to S$ as in \ref{fib-triv-nonzero-KS} and the quotient $\XX:=(X\times S')/\mu_p$ by diagonal action of $\mu_p$, the induced fibrewise trivial family $f\colon\XX\to S$ is nowhere isotrivial - one uses \cite[Prop. 6]{kollarmoduli} to show that $\I=\Isomm^{S\times X,\XX}$ is finite over $S$ and, using the same technique, one can prove that isotriviality would in fact imply that $S\times X\simeq \XX$. However, we don't know if its Kodaira--Spencer map is nowhere vanishing.
	\end{remark}
\newcommand{\etalchar}[1]{$^{#1}$}

\end{document}